\documentclass[12pt]{article}

\usepackage{amsmath}
\usepackage{amsfonts}
\usepackage{amsthm}
\usepackage{mathrsfs}
\usepackage{graphicx}
\usepackage{framed}
\usepackage{enumitem}
\usepackage[pagewise, displaymath, mathlines]{lineno}
\usepackage{fancyhdr,lipsum}

\def\bibname{}

\newtheorem{theorem}{Theorem}[section]
\newtheorem{lemma}[theorem]{Lemma}
\newtheorem{corollary}[theorem]{Corollary}
\newtheorem{proposition}[theorem]{Proposition}
\newtheorem{definition}[theorem]{Definition}

\def\R{\mathbb R}
\def\N{\mathbb N}

\def\<{\langle}
\def\>{\rangle}

\def\dist{{\rm dist}}

\newcommand{\romd}{\mathrm{d}}

\newcommand{\romdist}{\mathrm{dist}}
\newcommand{\lspan}{\mathrm{span}}

\newcommand*\patchAmsMathEnvironmentForLineno[1]{%
  \expandafter\let\csname old#1\expandafter\endcsname\csname #1\endcsname
  \expandafter\let\csname oldend#1\expandafter\endcsname\csname end#1\endcsname
  \renewenvironment{#1}%
     {\linenomath\csname old#1\endcsname}%
     {\csname oldend#1\endcsname\endlinenomath}}% 
\newcommand*\patchBothAmsMathEnvironmentsForLineno[1]{%
  \patchAmsMathEnvironmentForLineno{#1}%
  \patchAmsMathEnvironmentForLineno{#1*}}%
\AtBeginDocument{%
\patchBothAmsMathEnvironmentsForLineno{equation}%
\patchBothAmsMathEnvironmentsForLineno{align}%
\patchBothAmsMathEnvironmentsForLineno{flalign}%
\patchBothAmsMathEnvironmentsForLineno{alignat}%
\patchBothAmsMathEnvironmentsForLineno{gather}%
\patchBothAmsMathEnvironmentsForLineno{multline}%
}
  
\begin{document}

\title{Embedding properties of sets with finite box-counting dimension}

\author{Alexandros Margaris and James C. Robinson\\
   Mathematics Institute\\ University of Warwick, Coventry\\ CV4 7AL, UK}

\maketitle

\begin{abstract}
In this paper we study the regularity of embeddings of finite--dimensional subsets of Banach spaces into Euclidean spaces. In 1999, Hunt and Kaloshin [Nonlinearity 12 1263-1275] introduced the thickness exponent and proved an embedding theorem for subsets of Hilbert spaces with finite box--counting dimension. In 2009, Robinson [Nonlinearity 22 711-728] defined the dual thickness and extended the result to subsets of Banach spaces. Here we prove a similar result for subsets of Banach spaces, using the thickness rather than the dual thickness. We also study the relation between the box-counting dimension and these two thickness exponents for some particular subsets of $\ell_{p}$.
\end{abstract}

\section{Introduction}

The main question we want to address in this paper is how we can understand the notion that an arbitrary subset $X$ of a Banach space is `finite-dimensional'. There are many possible dimensions that we can consider and each of them provides `nice' embedding properties into Euclidean spaces. There have been embedding results concerning subsets with finite Hausdorff dimension by Ma\~n\'e \cite{Man}, finite box-counting dimension by Foias \& Olson \cite{FO}, finite doubling dimension by Assouad \cite{Ass}, Olson \& Robinson \cite{OR} and by Naor \& Neiman \cite{NN}. However, here we are particularly interested in subsets with finite upper box--counting dimension, whose definition we now recall. 

\begin{definition}\label{BC2}
Suppose that $(E,\|\cdot\|)$ is a normed space. Let $X$ a compact subset of $E$ and let $N(X, \epsilon)$ denote the minimum number of balls of radius $\epsilon$ with centres in $X$ required to cover $X$. The upper box-counting dimension of $X$ is 

\begin{equation}\label{BC3}
\romd_{B}(X) = \limsup_{\epsilon \rightarrow 0} \frac{\log N(X, \epsilon)}{-\log \epsilon}.
\end{equation}

\end{definition}

It follows from the definition that if $d > \romd_{B}(X)$, then there exists some positive constant $C=C_{d}$, such that 

\begin{equation}
N(X, \epsilon) \leq C \epsilon^{-d}.
\end{equation}

For the rest of the paper, we will refer to $\romd_{B}(X)$ as the box--counting dimension of $X$.

In 1996, Foias and Olson treated the case where $X$ is a subset of a Hilbert space $H$ with finite box--counting dimension and proved that there exists a Lipschitz map $L \colon H \to \mathbb{R}^{k}$ with a H\"older inverse on the image of $X$. In 1999, Hunt and Kaloshin established the existence of a `large' set of linear maps $L \colon H \to \mathbb{R}^{k}$ with H\"older continuous inverses on the image of $X$. In order to do so, they introduced a new quantity, called the thickness exponent of $X$, which measures how well an arbitrary subset of a Banach space can be approximated by finite--dimensional linear subspaces. We note that all Banach spaces we mention from now on are real.

\begin{definition}
Let $X$ be a subset of a Banach space $\mathfrak{B}$. The thickness exponent of $X$ in $\mathfrak{B}$, $\tau(X, \mathfrak{B})$ is defined as:

\[ \tau(X, \mathfrak{B}) = \limsup_{\epsilon \rightarrow 0} \frac{\log d(X,\epsilon)}{-\log\epsilon},\]
where $d(X,\epsilon)$ denotes the smallest dimension of those linear subspaces $V$ that satisfy
\[\romdist_{\mathfrak{B}}(x,V) \leq \epsilon \, \, \mbox{for all} \, \, x \, in \, X.\]
If no such subspace exists, we set $d(X, \epsilon) = \infty$ and we adopt a similar convention throughout this paper.

\end{definition}
Note that when $\tau > \tau(X)$, then there exists some positive constant $C$ such that

\[d(X, \epsilon) \leq C \epsilon^{-\tau}.\]

It is easy to see that the thickness exponent is always bounded above by the box--counting dimension. 
Indeed, if we cover $X$ by $N(X,\epsilon)$ balls of radius $\epsilon$ and let $V$ be the subspace of $\mathfrak{B}$ that is spanned by the centres of these balls, then every element of $X$ is within $\epsilon$ of $V$.

Moreover, Hunt and Kaloshin introduced a probability measure $\mu$ defined on a particular class of linear maps $L \colon H \to \mathbb{R}^{k}$. Using this measure $\mu$ and the above exponent, they proved the following embedding theorem for subsets of Hilbert spaces.

\begin{theorem}[Hunt \& Kaloshin, 1999]\label{ET1}
Let $X$ be a compact subset of a real Hilbert space $H$ with thickness exponent $\tau(X)$ and box--counting dimension $\romd_{B}(X) < \infty$. Then for any integer $k > 2d_{B}(X)$ and any given $\theta$ with
\[ 0 < \theta < \frac{k -2d_{B}(X)}{k\left(1 + \frac{\tau(X)}{2}\right)},\]
$\mu$--almost every linear map $L \colon H \to \mathbb{R}^{k}$ satisfies
\begin{equation}\label{3.10}
\|x-y\| \leq C_{L}|Lx-Ly|^{\theta}, \, \, \, \forall \, \, \, x,y \in X, \,\, \mbox{for some}\,\, C_{L}>0.
\end{equation}
In particular, every such $L$ is bijective from $X$ onto $L(X)$ with a H\"older continuous inverse.
\end{theorem}

By using the fact that the thickness is bounded above by the box--counting dimension, the above theorem can be restated such that the range of the exponent $\theta$ depends solely on the box--counting dimension.

The authors attempted to extend the theorem for subsets of Banach spaces and their proof relies on the claim that there exists a linear isometry from the dual of any finite--dimensional subspace of $\mathfrak{B}$ to a linear subspace of the dual of $\mathfrak{B}$. However, Kakutani \cite{Kak} proved that this can only be true in the context of a Hilbert space. To circumvent this problem, Robinson \cite{Rob} introduced a new exponent, the `dual thickness' which was defined based on an approximation required in the course of Hunt and Kaloshin's argument.  

\begin{definition}
Suppose that $X$ is a subset of a Banach space $\mathcal{B}$ and for every $\epsilon,\theta > 0$ let 
$\romd_{\theta}\left(X,\epsilon\right)$ denote the minimum dimension of all those subspaces $U$ of  $\mathcal{B}^{*}$ with the property that for every $x,y \in X$ with $\|x-y\|\geq \epsilon$, there exists some $\phi \in U$, such that

\[\left|\phi(x-y)\right| \geq \epsilon^{1+\theta}.\]
Then, for every $\theta >0$, we define
\[\tau_{\theta}^{*}\left(X\right) = \limsup_{\epsilon \rightarrow 0}\frac{\log\romd_{\theta}\left(X,\epsilon\right)}{-\log\epsilon},\] and then set
\[\tau^{*}\left(X\right) = \lim_{\theta \rightarrow 0} \tau_{\theta}^{*}\left(X\right).\]
\end{definition}
This admittedly unwieldy definition allows for the following result.

\begin{theorem}[Robinson, 2009]\label{ET2}
Let $X$ be a compact subset of a Banach space $\mathfrak{B}$ with dual thickness exponent $\tau^{*}(X)<\infty$ and box--counting dimension $\romd_{B}(X)<\infty$. Then for any integer $k > 2d_{B}(X)$ and any given $\theta$ with
\[0 < \theta < \frac{k -2d_{B}(X)}{k\left(1 + \tau^{*}(X)\right)},\]
$\mu$--almost every linear map $L \colon \mathfrak{B} \to \mathbb{R}^{k}$ satisfies

\begin{equation}\label{3.10}
\|x-y\| \leq C_{L}|Lx-Ly|^{\theta}, \, \, \, \forall \, \, \, x,y \in X, \,\, \mbox{for some}\,\, C_{L}>0.
\end{equation}
In particular, every such $L$ is bijective from $X$ onto $L(X)$ with a H\"older continuous inverse.
\end{theorem}

In Section \ref{S4} we prove that the dual thickness is bounded above by twice the box--counting dimension of $X$ which gives a range of $\theta$ independent of the dual thickness.
In particular, for any \begin{equation}\label{Range}0 < \theta < \frac{1}{1 + 2\romd_{B}(X)},
\end{equation} we can find an embedding into $\mathbb{R}^{k}$, for large enough $k$, such that the inverse is $\theta$-H\"older continuous. 

There is no known general relation between the thickness and the dual thickness in the context of a Banach space. However, Robinson \cite{JCR} proved that zero thickness implies zero dual thickness, which immediately implies that subsets of Banach spaces with $\tau(X) = 0$ admit embeddings for any positive exponent $\theta <1$. In Section \ref{S3}, we establish the same embedding result for $\tau(X)=0$ directly, i.e. without having to use the dual thickness.

In Section \ref{S2}, we consider $X$ as a compact subset of a Banach space with finite box-counting dimension and provide an embedding into a Hilbert space with a bound on the H\"older exponent of the inverse that depends on the box--counting dimension of $X$. As a corollary of this argument and Theorem \ref{ET1} we immediately obtain an embedding into an Euclidean space that does not require the dual thickness. We also extend the result to any compact metric space using the Kuratowski embedding.

Then, in Section \ref{S3}, we use the techniques introduced by Hunt and Kaloshin along with some key new arguments and prove the following embedding theorem for compact subsets of Banach spaces with thickness exponent less than 1.

\begin{theorem}\label{ET5}
Let $X$ be a compact subset of a Banach space $\mathcal{B}$ with thickness exponent $\tau(X) < 1$ and box--counting dimension $\romd_{B}(X)< \infty$. Then for any integer $k > 2d_{B}(X)$ and any given $\theta$ with
\[ 0 < \theta < (1-\tau(X)) \frac{k -2d_{B}(X)}{k\left(1 + \tau(X)\right)},\]
$\mu$--almost every linear map $L \colon \mathcal{B} \to \mathbb{R}^{k}$ satisfies:
\begin{equation}\label{3.10}
\|x-y\| \leq C_{L}|Lx-Ly|^{\theta}, \, \, \, \forall \, \, \, x,y \in X, \,\, \mbox{for some}\,\, C_{L}>0.
\end{equation}
In particular, $L$ is bijective from $X$ onto $L(X)$ with a H\"older continuous inverse.
\end{theorem}

We note that when $\tau(X) = 0$, given any $0 < \theta <1$, $X$ admits finite--dimensional embeddings with a $\theta$-H\"older continuous inverse, exactly as in the previous two theorems.
We also note that the above result provides a bigger range for $\theta$ comparing to \eqref{Range}, whenever 
\[\frac{1-\tau(X)}{1 + \tau(X)} > \frac{1}{1 + 2\romd_{B}(X)} \Leftrightarrow \tau(X) < \frac{\romd_{B}(X)}{1 + \romd_{B}(X)}.\]
The restriction on the thickness here is surprising and it is an interesting open problem whether there is a result that extends Theorem \ref{ET5} for thickness $\tau \geq 1$. 

Finally, in Section \ref{S4} we look closely at the thickness and dual thickness and how they relate to the box--counting dimension. We prove some general estimates and we also look at a particular class of sequences in $\ell_{p}$ which was used by Pinto de Moura and Robinson \cite{DR} to prove that Theorem \ref{ET2} is asymptotically sharp as $k \rightarrow \infty$.

\section{Embeddings from Banach into \\Hilbert spaces}\label{S2}

In this section, we prove two embedding results into a Hilbert space. Both of them can be combined with Theorem \ref{ET1} to provide an embedding theorem for compact subsets of Banach spaces into Euclidean spaces that does not require the arguments in Robinson's result \cite{Rob}. Before we embark on the proofs of the embedding theorems, we recall some elementary but useful properties of the box--counting dimension, which we will be using in what follows. The proofs can be found in Robinson \cite{JCR}.

\begin{lemma}\label{2}\hfill
Suppose $(E_{1}, \| \cdot \|_{1})$ and $(E_{2}, \| \cdot \|_{2})$ are normed spaces.
\begin{enumerate}
\item If $f \colon (E_{1}, \|\cdot \|_{1}) \to (E_{2}, \|\cdot \|_{2})$ is a Lipschitz function, i.e. there exists a constant $C >0$, such that
\[\|f(x_{1}) - f(x_{2})\|_{2} \leq C \|x_{1} - x_{2}\|_{1} \,\,\, \text{for all} \, \, \, x_{1}, x_{2} \in E_{1},\] 
then for all compact subsets $X \subset E_{1}$ we have
\[\romd_{B}(f(X)) \leq \romd_{B}(X).\]
\item Let $E_{1} \times E_{2}$ be the product space equipped with some product metric. Suppose also that $X\subseteq E_{1}$ and $Y \subseteq E_{2}$ are compact. Then,
\[ \romd_{B}(X \times Y) \leq \romd_{B}(X) + \romd_{B}(Y).\]
\end{enumerate}
\end{lemma}

\subsection{Embedding when $\romd_{B}(X)$ is finite}

We first show that any compact subset of a Banach space with finite box--counting dimension embeds into a Hilbert space. 

\begin{proposition}\label{ETH1}
Suppose that $X$ is a compact subset of a Banach space $\mathfrak{B}$ with finite box-counting dimension. Then, for every $\alpha > 1+ \romd_{B}(X)$ there exists a linear map $\Phi \colon \mathfrak{B} \to H$, where $H$ is a separable Hilbert space, such that for every $x,y \in X$
\[C_{\alpha}^{-1} \|x-y\|^{\alpha} \leq \left|\Phi(x) - \Phi(y)\right| \leq C_{\alpha}\|x-y\|,  \, \mbox{for some} \,\, C_{\alpha}>0.\]
\end{proposition}

We first prove that for every $n \in \mathbb{N}$ there exists a linear embedding $\phi_{n}$ into an Euclidean space $\mathbb{R}^{m_{n}}$ such that $\phi_{n}^{-1}$ satisfies a Lipschitz condition for all $x,y \in X$ with $\|x-y\| \geq 2^{-n}$. 

\begin{lemma}\label{1EL}
Suppose that $X$ is as above. Then, given $d > \romd_{B}(X)$ and $n \in \mathbb{N}$, there exist $\phi_{n} \in \mathcal{L}\left(\mathcal{B} ; \mathbb{R}^{m_{n}}\right)$, where $m_{n} \leq C 2^{2n d}$, such that $\|\phi_{n}\| \leq \sqrt{m_{n}}$ and 
\[\left| \phi_{n}(x-y)\right| \geq 2^{-(n+1)} \, \, \mbox{whenever} \,\,  \|x-y\| \geq 2^{-n}, \,\mbox{for x,y  in} \, X.\]
\end{lemma}

\begin{proof}
Let $Z =X-X = \{x-y:x,y \in X\}.$ Then, it is easy to see that $\romd_{B}(Z) \leq 2\romd_{B}(X).$ 
Indeed, $Z$ is the image of $X \times X$ under the Lipschitz map 
\[(x,y) \mapsto x-y\] and so by Lemma \ref{2}, we obtain
\[\romd_{B}(f(X \times X)) = \romd_{B}(X-X) \leq \romd_{B}(X \times X) \leq 2\romd_{B}(X).\]

Given $d$ as in the statement of the lemma, we can cover $Z$ by no more than $m_{n} = N(Z, 2^{-(n+2)}) \leq C_{d} 2^{2nd}$ balls of radius $2^{-(n+2)}$.
Let the centres of these balls be $z_{i}$. 
By the Hahn--Banach Theorem, we can find $f_{i} \in \mathcal{B}^{*}$ such that $\|f_{i}\| = 1$ and $f_{i}(z_{i}) = \|z_{i}\|$. 
Now, define $\phi_{n} \colon \mathcal{B} \to \mathbb{R}^{m_{n}}$ as
\[\phi_{n}(x) = \left(f_{1}(x),...,f_{m_{n}}(x)\right).\]
It is immediate that $\|\phi_{n}\| \leq \sqrt{m_{n}}$.

Suppose now that $z \in X-X$ such that $\|z\| \geq 2^{-n}$. Then, there exists some $i \leq m_{n}$ such that $\|z-z_{i}\| \leq 2^{-(n+2)}$, and therefore
\begin{align*}
\left|\phi_{n}(z)\right| & \geq |f_{i}(z)| = |f_{i}(z_{i}) + f_{i}(z-z_{i})|\\
& \geq \|z_{i}\| - \|z-z_{i}\| \geq \|z\| -2\|z-z_{i}\|\\
& \geq 2^{-n} - 2^{-(n+1)} \geq 2^{-(n+1)}.\qedhere
\end{align*}
\end{proof}
We now continue with the proof of Proposition \ref{ETH1}.

\begin{proof}[Proof of Proposition \ref{ETH1}]
We first construct a new separable Hilbert space $H$ given an orthonormal basis $(e_i)_{i=1}^\infty$ of $\ell^2$ and a sequence $(m_i)_{i=1}^\infty$ of positive integers,
by taking the collection
\[
e_i\otimes w^{m_i}_j\qquad i\in\N,\ j=1,\ldots,m_i,
\]
as an orthonormal basis for $H$, where $(w_{j}^N)_{j=1}^N$ denotes an orthonormal basis for $\R^N$. We define the inner product $\<\cdot,\cdot\>$ on $H$ to ensure that this is indeed an orthonormal set, i.e.\ we set
\[
\<e_i\otimes w^{m_i}_j,e_{i'}\otimes w^{m_{i'}}_{j'}\>=\delta_{ii'}\delta_{jj'}.
\]
In particular if $x_i\in\R^{m_i}$, $i\in\N$, then
\[
\left\|\sum_{i=1}^\infty e_i\otimes x_i\right\|_H^2=\sum_{i=1}^\infty\|x_i\|_{\R^{m_i}}^2.
\]

Take $d = d_{\alpha}>0$ such that $\alpha >1+d> 1 + \romd_{B}(X)$. Then, for this $d > \romd_{B}(X)$, we consider $\phi_{n}, m_{n}$ given by Lemma \ref{1EL}, and then from the above construction we consider $H$ based on the sequence $(m_{n})_{n=1}^{\infty}$. 
Now, for $x \in \mathfrak{B}$ we set
\[\Phi(x) = \sum_{n=1}^{\infty} 2^{(1-\alpha) n} \phi_{n}(x) \otimes e_{n} \in H.\]
Clearly $\Phi$ is linear and
\[\|\Phi\|^{2} \leq \sum_{n=1}^{\infty} 2^{2(1-\alpha) n} \|\phi_{n}\|^{2} \leq \sum_{n=1}^{\infty} 2^{2(1+d-\alpha)n} < \infty,\]
since $1+d-\alpha <0.$

Now, take any $x,y \in X$ and suppose $x \neq y$ (the case $x=y$ is trivial).
If $\|x-y\| \geq 1$, then it suffices to take $R >0$ such that 
\[X-X \subset B(0,R).\]
Therefore, using also that $\|\phi_{1}(x-y)\| \geq \frac{1}{4}$, we have that
\[\|\Phi(x-y)\| \geq 2^{1-\alpha} \|\phi_{1}(x-y)\| \geq \frac{2^{1-\alpha}}{4} \left(\frac{\|z\|}{R}\right)^{\alpha} = C_{\alpha} \|z\|^{\alpha}.\]

If $0 < \|x-y\| < 1$, consider $n$ such that 
\[2^{-n} \leq \|x-y\| < 2^{-(n-1)}.\]
Thus, we obtain
\begin{align*}
\|\Phi(x-y)\| & \geq 2^{(1-\alpha)n} |\phi_{n}(x-y)|\\
& \geq 2^{(1-\alpha)n} 2^{-n-1} \geq C_{\alpha} \, 2^{-\alpha n+1}\\
& \geq C_{\alpha} \|x-y\|^{\alpha}.\qedhere
\end{align*}
\end{proof}
By combining Proposition \ref{ETH1} and Theorem \ref{ET1}, we can now obtain an embedding theorem for compact subsets of Banach spaces into finite--dimensional spaces. The difference here is that the range of the exponent depends on the thickness and the box--counting dimension rather than the dual thickness.

\begin{theorem}\label{ET3}
Let $X$ be a compact subset of a Banach space $\mathfrak{B}$ with thickness exponent $\tau(X)$ and box--counting dimension $\romd_{B}(X)$. Then for any integer $k > 2d_{B}(X)$ and any given $\theta$ with
\[0 < \theta < \frac{k -2d_{B}(X)}{k\left(1 + \romd_{B}(X)\right)\left(1 + \frac{\tau(X)}{2}\right)},\]
there exists a linear map $L \colon \mathfrak{B} \to \mathbb{R}^{k}$ such that
\begin{equation}\label{3.10}
\|x-y\| \leq C_{L}|Lx-Ly|^{\theta}, \, \, \forall \, \, \, x,y \in X.
\end{equation}
In particular, $L$ is bijective from $X$ onto $L(X)$ with a H\"older continuous inverse.
\end{theorem}

\begin{proof}
Take $\theta_{1}$ such that
\[\frac{\theta k\left(1 + \frac{\tau(X)}{2}\right)}{k -2d_{B}(X)}  < \theta_{1} <  \frac{1}{1 + \romd_{B}(X)},\]
and set \[\theta_{2} = \frac{\theta}{\theta_{1}} < \frac{k -2d_{B}(X)}{k\left(1 + \frac{\tau(X)}{2}\right)}.\]

By Proposition \ref{ETH1} (substituting $\alpha = \theta_{1}^{-1}$), there exists a separable Hilbert space $H$ and a linear map $\Phi \colon \mathfrak{B} \to H$ such that for every $x,y \in X$
\[C_{\theta_{1}}^{-1} \|x-y\| \leq \left|\Phi(x) - \Phi(y)\right|^{\theta_{1}},  \, \mbox{for some} \,\, C_{\theta_{1}}>0.\]

We know from Lemma \ref{2} that the box--counting dimension of $X$ does not increase under $\Phi$. 
We now check that the same holds for the thickness exponent of $X$. 
Take $\epsilon >0$ and let $V$ be the linear subspace of $\mathfrak{B}$ with the smallest dimension among all those that satisfy 
\[\dist(x,V) < \epsilon,\] for all $x \in X$.
Let $y = \Phi(x) \in \Phi(X)$ and if we let $v \in V$ such that
\[\|v - x\| < \epsilon,\] then 
\[\|y-\Phi(v)\| \leq \|\Phi\| \epsilon.\]
Since $\Phi(V)$ is a linear subspace of $H$ and $\dim(\Phi(V)) \leq \dim(V)$, we have
\begin{align*}
\tau(\Phi(x)) & = \limsup_{\epsilon \rightarrow 0} \frac{\log d_{H}\left(\Phi(X), \|\Phi\| \epsilon\right)}{-\log \|\Phi\| \epsilon}\\
& \leq  \limsup_{\epsilon \rightarrow 0} \frac{\log d_{\mathfrak{B}}\left(X, \epsilon\right)}{-\log \|\Phi\| \epsilon} = \tau(X).
\end{align*}
Since \[\theta_{2} = \frac{\theta}{\theta_{1}} < \frac{k -2d_{B}(X)}{k\left(1 + \frac{\tau(X)}{2}\right)} \leq \frac{k -2d_{B}(X)}{k\left(1 + \frac{\tau(\Phi(X))}{2}\right)},\] by Theorem \ref{ET2} there exists a linear map $T \colon H \to \mathbb{R}^{k}$ and a positive constant $C_{\theta}$ such that
\[\|x-y\| \leq C_{\theta} |T(\Phi(x)) - T(\Phi(y))|^{\theta_{1}\theta_{2}},\]
for all $x,y \in X$.
We conclude the proof by setting $L= T \circ \Phi$.
\end{proof}

Note that the above theorem gives an embedding for all compact subsets of Banach spaces with finite box--counting dimension, since $\tau(X) \leq \romd_{B}(X)$. The result improves on the range of $\theta$ from Theorem \ref{ET2} (see \eqref{Range}) whenever
\[\tau(X) < \frac{2\romd_{B}(X)}{1 + \romd_{B}(X)}.\]
However, we note that when $\tau(X) = 0$, we obtain $\theta$-H\"older embeddings for any 
\[0 < \theta < \frac{1}{1 + \romd_{B}(X)},\]
which is not optimal.

The above embedding into a Banach space can also be used as a tool to prove an embedding theorem for compact metric spaces with finite box--counting dimension. We first recall the Kuratowski embedding theorem, which allows us to isometrically embed any compact metric space into a Banach space.

\begin{lemma}[Kuratowski Embedding]\label{ET100}
If $(X,d)$ is any compact metric space, then the map
\[x \mapsto \Phi (x) = d(\cdot,x)\]
is an isometry of $(X,d)$ onto a subset of $L^{\infty}(X)$.
\end{lemma}

For a proof of the above result, see Heinonen \cite{HJ}.
Hence, we have the following corollary of Theorem \ref{ET3} and Lemma \ref{ET100}, using the fact that
\[\tau(\Phi(X)) \leq \romd_{B}(\Phi(X)) = \romd_{B}(X),\]
where $\Phi$ is the isometry from Lemma \ref{ET100} and $X$ is an arbitrary compact metric space.

\begin{corollary}
Suppose $(X,d)$ is a compact metric space with finite box--counting dimension. Then, for any $k > 2\romd_{B}(X)$ and any given $\theta$ with
\[ 0 < \theta < \frac{k -2d_{B}(X)}{k\left(1 + \romd_{B}(X)\right)\left(1 + \frac{\romd_{B}(X)}{2}\right)},\]
there exists a Lipschitz map $\psi \colon (X,d) \to \mathbb{R}^{k}$ such that
\[\|x-y\| \leq C_{\psi}|\psi(x)-\psi(y)|^{\theta}, \, \, \forall \, \, \, x,y \in X.\]
\end{corollary}

The range of $\theta$ in the above Corollary improves on the respective range in the paper by Foias and Olson \cite{FO}. There, the authors use a direct argument to prove that if $d = \max\{1, \romd_{B}(X)\}$, then for any given
\[\theta < \frac{1}{2d \left(1 + \frac{\romd_{B}(X)}{2}\right)}\] any compact metric space with finite box-counting dimension can be embedded into a sufficiently large Euclidean space such that the inverse is $\theta$-H\"older continuous.

\subsection{Embedding when $\tau(X) < 1$}

We now prove another embedding into a Hilbert space, in which the range of the H\"older exponent depends solely on the thickness exponent of $X$. This result also provides some motivation towards the next section.

\begin{proposition}\label{ETH2}
Suppose that $X$ is a compact subset of a Banach space $\mathfrak{B}$ with thickness exponent $\tau(X) <1$. Then, for every \[\alpha > \frac{1+ \tau(X)}{1 -\tau(X)}\] there exists a separable Hilbert space $H$ and a  linear map $\Phi \colon \mathcal{B} \to H$, such that
\[C_{\alpha}^{-1} \|x-y\|^{\alpha} \leq \left|\Phi(x) - \Phi(y)\right| \leq C_{\alpha}\|x-y\|, \, \mbox{for all} \,\, x,y \, \, \mbox{in} \, \,  X.\]
\end{proposition}
Following the previous procedure, we first need the following lemma.

\begin{lemma}\label{ETH3}
Suppose that $\tau(X) <1.$ Then, for every $1> \tau > \tau(X)$, there exists some $\beta_{\tau}>1$ such that for every $n \in \mathbb{N}$, we can find $\phi_{n} \in \mathcal{L}\left(\mathcal{B} ; \mathbb{R}^{m_{n}}\right)$, $C_{\beta_{\tau}} > 0$, where $m_{n} \leq C_{\beta_{\tau}} 2^{\beta_{\tau} n k}$, with $\|\phi_{n}\| \leq \sqrt{m_{n}}$ and 
\[\left| \phi_{n}(x-y)\right| \geq 2^{-(\beta_{\tau} n+1)}, \, \, \text{whenever} \,\,  \|x-y\| \geq 2^{-n}.\]
\end{lemma}

Before we prove the above lemma, we recall the definition of an Auerbach basis for a finite--dimensional Banach space.
\begin{definition}
Suppose that $U$ is a finite--dimensional Banach space. 
An Auerbach basis for $U$ is formed by a basis $\{e_{1},...,e_{n}\}$ of $U$ coupled with corresponding elements $\{f_{1},...f_{n}\}$ of $U^{*}$ that satisfy $\|f_{i}\|=\|e_{i}\|=1$ and 
\[f_{i}(e_{j})=\delta_{ij}.\]
\end{definition}
For a proof of the existence of such a basis, see Exercise 7.3 in the book of Robinson \cite{JCR}, for example.

\begin{proof}[Proof of Lemma \ref{ETH3}]
Take any $\beta >1$. Then, by the definition of the thickness exponent, there exists a subspace $V_{n}$ of $\mathcal{B}$ and some $C=C_{\beta}>0$, such that $\dim(V_{n})=m_{n} \leq C 2^{\beta n \tau}$ and 
\[\mathrm{dist}(x,V_{n}) \leq 2^{-\beta n -2}.\]

Suppose that $\{u^{n}_{1},..,u^{n}_{m_{n}}\}$ is an Auerbach basis for $V_{n}$, and let $\{f^{n}_{1},..,f^{n}_{m_{n}}\}$ be the corresponding elements of $V_{n}^{*}$ that satisfy $\|f^{n}_{i}\| = 1 ,\, \forall i$ and
\[f^{n}_{i}(u^{n}_{j}) = \delta_{ij}.\]

We now define a projection $P_{n}$ onto $V_{n}$ as
\[P_{n}(x) = \sum_{i=1}^{m_{n}} f^{n }_{i}(x) u^{n}_{i}\]
and define $\phi_{n} \colon \mathfrak{B} \to \mathbb{R}^{m_{n}}$ by setting
\[\phi_{n}(x) = (f^{n}_{1}(x),...,f^{n}_{m_{n}}(x)).\]

Obviously $\|\phi_{n}\| \leq \sqrt{m_{n}} \leq 2^{\beta n \tau/2}$. Moreover, let 
$z \in X-X$ be such that $\|z\| \geq 2^{-n}$ and choose $z_{n} \in V_{n}$ such that
\[\|z-z_{n}\| \leq 2^{-\beta n -2}.\]
Then 
\[\|z_{n}\| \geq 2^{-n} - 2^{-\beta n -2} \geq 2^{-n} - 2^{-n-2} \geq 2^{-n-1}.\]
Now, write $z_{n} = \sum_{i=1}^{m_{n}} z_{n}^{i} u^{n}_{i}$ and take $j \leq m_{n}$ such that $\|(z_{n}^{1},...,z_{n}^{m_{n}})\|_{\infty}= |z_{n}^{j}|$.
Then,
\begin{align*}
|\phi_{n}(z)\|_{2} & \geq |f^{n }_{j}(z)| \geq |f^{n }_{j}(z_{n})| - |f^{n }_{j}(z-z_{n})|\\
& \geq |z_{n}^{j}| - \|z-z_{n}\| \geq m_{n}^{-1} \|z_{n}\| - 2^{-\beta n-2}\\
& \geq C 2^{-\beta n \tau} 2^{-n} - 2^{-\beta n-2} = C 2^{-n(1+\beta \tau)} - \frac{1}{4}2^{-\beta n}.
\end{align*}
Now, we choose $\beta = \beta_{\tau}$ such that
\[1+ \beta_{\tau} \tau = \beta_{\tau} \Leftrightarrow \beta_{\tau} = \frac{1}{1-\tau} >1,\]
which concludes the proof.
\end{proof}

We now prove Proposition \ref{ETH2} .

\begin{proof}[Proof of Proposition \ref{ETH2}]
Take $1 >\tau > \tau(X)$ such that
\[\alpha > \frac{1+\tau}{1-\tau} = \beta_{\tau}+ \tau\beta_{\tau},\] and let $\phi_{n}, m_{n}$ be as given in the previous lemma. 

Now, let $(e_{n})_{n=1}^{\infty}$ be the standard basis for $\ell_{2}$ and following the construction in the proof of Proposition \ref{ETH1} we define $H$ based on the sequence $(m_{n})_{n=1}^{\infty}$. We now set 
\[\Phi(x) = \sum_{n=1}^{\infty} 2^{(\beta_{\tau}-\alpha) n} \phi_{n}(x)\otimes e_{n} \in H.\]
Then,
\[\|\Phi\| \leq \sum_{n=1}^{\infty} 2^{(\beta_{\tau}-\alpha) n}  2^{\tau \beta_{\tau}n} < \infty.\]
Now, take any $x,y \in X$ with $x \neq y$.
If $\|x-y\| \geq 1$, we argue exactly as in the proof of Proposition \ref{ETH1}.
If $0 < \|x-y\| <1$, let $n$ such that $2^{-n} \leq \|x-y\| < 2^{-n+1}.$ 

Therefore
\begin{align*}
\|\Phi(x-y)\| & \geq 2^{(\beta_{\tau}-\alpha) n} \|\phi_{n}(x-y)\|_{2}\\
& \geq 2^{(\beta_{\tau}-\alpha) n} 2^{-\beta_{\tau} n -1} \\
& \geq  2^{-\alpha n -1} \geq C_{\alpha} \|x-y\|^{\alpha}.\qedhere
\end{align*}
\end{proof}

Just as in the previous situation, we can now obtain a linear embedding from a compact subset of a Banach space with finite box-counting dimension into a finite--dimensional space such that the inverse is $\theta$-H\"older continuous for any 
\[0< \theta < \frac{1 - \tau(X)}{(1 + \tau(X))\left(1 + \frac{\tau(X)}{2}\right)}.\] However, in the next section, we give a more direct proof that not only improves this exponent, but also provides a set of embeddings with `full measure'.

\section{Embedding subsets of Banach spaces into $\mathbb{R}^{k}$}\label{S3}

Before we prove our main embedding result, we will recall, following Robinson \cite{JCR}, the construction of a certain probability measure that is based on the ideas in Hunt and Kaloshin \cite{HK} and will play a key role in our proof.

\subsection{A measure based on sequences of linear subspaces}\label{M}

Suppose that $\mathfrak{B}$ is a Banach space and $\mathcal{V} = \{V_{n}\}_{n=1}^{\infty}$ a sequence of finite--dimensional subspaces of $\mathfrak{B}^{*}$, the dual of $\mathfrak{B}$.
Let us denote by $d_{n}$ the dimension of $V_{n}$ and by $B_{n}$ the unit ball in $V_{n}$. 

Now, we fix a real number $\alpha >1$ and define the space $\mathbb{E}_{\alpha}(\mathcal{V})$ as the collection of linear maps $L \colon \mathfrak{B} \to \mathbb{R}^{k}$ given by 
\[\mathbb{E} = \mathbb{E}_{\alpha}(\mathcal{V}) = \left\lbrace L = (L_{1}, L_{2},..., L_{k}) : L_{i} = \sum_{n=1}^{\infty} n^{-\alpha} \phi_{i,n} , \, \, \phi_{i,n} \in B_{n}\right\rbrace.\]
Let us also define
\[\mathbb{E}_{0} = \left\lbrace \sum_{n=1}^{\infty} n^{-\alpha} \phi_{i,n} , \, \, \phi_{i,n} \in B_{n}\right\rbrace.\]
Clearly $\mathbb{E} = \left(\mathbb{E}_{0}\right)^{k}$.

To define a measure on $\mathbb{E}$, we first take a basis for $V_{n}$ so that we can identify $B_{n}$ with a symmetric convex set $U_{n} \subset \mathbb{R}^{d_{n}}$. Then, we construct each $L_{i}$ randomly by choosing each $\phi_{i,n}$ with respect to the normalised $d_{n}$--dimensional Lebesgue measure $\lambda_{n}$ on $U_{n}$. Finally, by taking $k$ copies of this measure we obtain a measure on $\mathbb{E}.$
In particular we first consider $\mathbb{E}_{0}$ as a product space
\[\mathbb{E}_{0} = \prod_{n=1}^{\infty} B_{n},\] and define a measure $\mu_{0}$ on $\mathbb{E}_{0}$ as
\[\mu_{0} = \otimes_{n=1}^{\infty} \lambda_{n}.\]
Secondly, we consider  $\mathbb{E} = \mathbb{E}_{0}^{k}$ and define $\mu$ on $\mathbb{E}$ as
\[\mu = \prod_{i=1}^{k} \mu_{0}.\]

Hunt and Kaloshin \cite{HK} proved the following upper bound on

\[\mu\{ L \in \mathbb{E} : |Lx| \leq \epsilon\},\]
for $x \in \mathcal{B}$ and any $\epsilon > 0.$ For a more detailed proof, see Robinson \cite{JCR}.

\begin{lemma}\label{1.6}
Suppose that $x \in \mathcal{B}$, $\epsilon > 0, a \in \mathbb{R}$ and $\mathcal{V}=\{V_{n}\}$ as above. Then
\[\lambda_{n}\{\phi \in B_{n} : |a +  \phi(x) | \leq C \epsilon \} \leq  d_{n} \frac{\epsilon}{|g(x)|},\]
for any $g \in B_{n}$.
\end{lemma}

We are now ready to give the proof of Theorem \ref{ET5}. 

\begin{theorem}\label{ET5}
Let $X$ be a compact subset of a Banach space $\mathfrak{B}$ with thickness exponent $\tau(X) < 1$ and box--counting dimension $\romd_{B}(X)< \infty$. Then for any integer $k > 2d_{B}(X)$ and any given $\theta$ with
\[ 0 < \theta < (1-\tau(X)) \frac{k -2d_{B}(X)}{k\left(1 + \tau(X)\right)},\]
$\mu$--almost every linear map $L \colon \mathcal{B} \to \mathbb{R}^{k}$ satisfies:
\begin{equation}\label{3.10}
\|x-y\| \leq C_{L}|Lx-Ly|^{\theta}, \, \, \, \forall \, \, \, x,y \in X.
\end{equation}
In particular, $L$ is bijective from $X$ onto $L(X)$ with a H\"older continuous inverse.
\end{theorem}

The proof follows closely the techniques introduced in Hunt and Kaloshin's argument with some key differences. In particular, we first use the thickness exponent to construct a sequence of finite--dimensional subspaces of $\mathfrak{B}$, that `approximate' $X$. Then, we use an Auerbach basis to define a sequence of finite--dimensional subspaces of the dual of $\mathfrak{B}$ and define our probability measure based on this sequence.
\begin{proof}
Clearly \eqref{3.10} holds if and only if 
\begin{equation}\label{3.11}
\|z\| \leq C_{L} |Lz|^{\theta} \, \, \, \forall \, \, \, z \in X-X.
\end{equation}

We want to bound the measure of linear maps that fail to satisfy \eqref{3.11} for some $z$ in a restricted subset of $X-X$.
Take $1 > \tau > \tau(X)$ and $d > \romd_{B}(X)$ such that
\begin{equation}\label{3.15}
0 < \theta < (1-\tau)\frac{k -2d}{k\left(1 + \tau\right)}.
\end{equation}

Take some $\beta >1$, which will be chosen later on and for every $n \in \mathbb{N}$, by definition of the thickness exponent, we can find a linear subspace $V_{n} \subset \mathcal{B}$ such that 
\begin{equation}\label{3.12}
\dim(V_{n}) \leq C_{\beta}2^{\theta n \tau \beta}
\end{equation}
and
\begin{equation}\label{3.13}
\romd(X,V_{n}) \leq \frac{2^{-\theta n \beta}}{3}.
\end{equation}

Using an Auerbach basis for $V_{n}$, along with the Hahn--Banach theorem, we construct a subspace $G_{n}$ of $\mathcal{B}^{*}$, as follows.
Suppose that \[\{e^{n}_{1},...,e^{n}_{d_{n}}\}\] is a basis for $V_{n}$ and 
\[\{r^{n}_{1},...,r^{n}_{d_{n}}\}\] is the corresponding basis for $V_{n}^{*}$, which satisfies:
\[\|r^{n}_{i}\| = 1, \, \, \, \forall \, \, \, i\] and 
\[r^{n}_{i}(e^{n}_{j})= \delta_{ij}, \, \, \, \forall \, \, \, i\neq j.\]

Using the Hahn--Banach theorem, we extend the elements $r_{1}^{n},...,r_{d_{n}}^{n} \in V_{n}^{*}$ to maps $f_{1}^{n},...,f_{d_{n}}^{n}$ in $\mathfrak{B}^{*}$ and set
\[G_{n} = \langle f^{n}_{1},..,f^{n}_{d_{n}}\rangle,\]
a subspace of $\mathcal{B}^{*}$, that is at most $d_{n}$--dimensional.

We now construct a measure based on the sequence $\mathcal{G} = \{G_{n}\}_{n=1}^{\infty}$, according to the definitions given in the beginning of this section.

In particular, if $S_{n}$ is the unit ball in $G_{n}$, we define
\[\mathbb{E} = \mathbb{E}_{2}\left(\mathcal{G}\right) =  \{L=(L_{1},..,L_{k}) : L_{i}=\sum_{n=1}^{\infty} n^{-2} \phi_{i,n} , \phi_{i,n} \in S_{n}\}.\]
Given the above construction, we now consider

\[Z_{n} = \{z \in X-X: \|z\| \geq 2^{-\theta n}\}\] and
\[Q_{n} = \{L \in E : |Lz| \leq 2^{-n}, \, \, \, \text{for some} \, \, \, z \in Z_{n}.\}.\]

Our goal is to bound the measure of $Q_{n}$ by something summable over $n$ and use the Borel--Cantelli Lemma. 

Using the fact that $\romd_{B}(X-X) \leq 2\romd_{B}(X)$, we cover $Z_{n}$ by $C 2^{2nd}$ closed  balls of radius $2^{-n}$. 
We observe that if $z$ is in the intersection of $Z_{n}$ with one of these balls, which we denote by $B(z_{0}, 2^{-n}),$ then

\[|Lz_{0}| \leq |Lz| + |L(z-z_{0})| \leq (1 + \|L\|) 2^{-n}.\]
But, $\|L\|$ is bounded uniformly for all $L \in \mathbb{E}$.
Indeed
\[\|L\|^{2} \leq \sum_{i=1}^{k} |L_{i}|^{2}\] and
\[|L_{i}|^{2} = \left\vert \sum_{n=1}^{\infty} n^{-2} \phi_{i,n} \right\vert^{2} \leq 
\sum_{n=1}^{\infty} n^{-4} = C<\infty.\]
Hence,
\[|Lz_{0}| \leq M2^{-n},\]
for some positive constant $M$, which holds for all $L \in \mathbb{E}$.

Now, we wish to bound the measure of $L \in \mathbb{E}$ that fail to satisfy \eqref{3.11}, for some $z$ in $ Y= Z_{n}\cap B(z_{0},2^{-n})$. From the above discussion, we have that

\begin{align*}
\mu\{L \in E: |Lz| \leq 2^{-n} \, \, \, \text{for some} \, \, \, z \in Y\} & \leq  \mu\{L \in E : |Lz_{0}| \leq M2^{-n} \}.
\end{align*}
Now, consider $z_{n} \in V_{n}$ such that $\|z_{n} - z_{0}\| \leq 2^{-\theta\beta n}/3.$ 
Therefore,
\[\|z_{n}\| \geq C 2^{-\theta n} - 2^{-\theta \beta n}/3 \geq C 2^{-\theta n}.\]

We now write $z_{n}$ as
\[z_{n} = \sum_{i=1}^{d_{n}} z^{i}_{n} e_{i}^{n},\]
and consider $j\leq d_{n}$ such that
\[z_{n}^{j} = \|(z_{n}^{1},...,z_{n}^{d_{n}})\|_{\infty}.\]
We now define 
\[g_{n} = f_{j}^{n},\]
which satisfies
\[\|g_{n}\| = 1 \, \, \, \text{and} \, \, \, |g_{n}(z_{n})|\geq d_{n}^{-1} \|z_{n}\|.\]

Hence,

\begin{align*}
|g_{n}(z_{0})| & \geq d_{n}^{-1} \|z_{n}\| - \|z_{n}-z_{0}\| \geq C 2^{-n\theta\beta\tau} 2^{-\theta n} - 2^{-\theta \beta n}/3\\
& = C2^{-n\theta(\beta\tau+1)} - 2^{-n\theta\beta}/3.
\end{align*}
We now choose $\beta$ such that $\beta\tau +1 = \beta \Longleftrightarrow \beta = 
\frac{1}{1-\tau},$ which gives that
\[|g_{n}(z_{0})| \geq C 2^{-n\theta \beta}.\]

Using Lemma \ref{1.6}, we obtain:
\begin{align*}
\mu\{L \in E : |Lz_{0}| \leq M2^{-n} \} & \leq \left( n^{2} \, d_{n} \frac{M 2^{-n}}{|g_{n}(z_{0})|}\right)^{k}\\
& \leq C \left(n^{2} 2^{n \beta \theta \tau} 2^{-n} 2^{\theta \beta n}\right)^{k}.
\end{align*}

Thus,
\begin{align*}
\mu(Q_{n}) & \leq C 2^{2nd} \left(n^{2} 2^{n \beta \theta \tau} 2^{-n} 2^{\theta \beta n}\right)^{k},
\end{align*}
so the sum $\sum_{n=1}^{\infty} \mu(Q_{n})$ is finite iff 
\[\theta < (1-\tau) \frac{k -2d_{B}(X)}{k\left(1 + \tau\right)}.\]

Thus, by the Borel--Cantelli lemma $\mu(\limsup Q_{n}) =0$, i.e.\ 
$\mu$--almost every $L$ lies in only a finite number of the $Q_{n}$. For such an $L$, there exists a $n_{L}$, such that for every $n \geq n_{L}$, $L$ does not belong to $Q_{n}$. In particular
\[ \text{if} \, \, \, |z| \geq 2^{-n\theta} \, \, \,  \text{then} \, \, \, |Lz| \geq 2^{-n}, \, \, \, \text{for all}
 \, \, n \geq n_{L}.\]

To complete the argument, we use the fact that $X-X$ is compact and we claim the existence of an $R >0$, such that $ X-X \subseteq B(0,R).$ 

Now, let $z \in X-X$ and consider the following cases
\[ \text{if} \, \, \, |z| \geq 2^{-n_{L}\theta},\]
then
\[|Lz| \geq 2^{-n_{L}} \geq \frac{2^{-n_{L}}}{R^{\frac{1}{\theta}}}|z|^{\frac{1}{\theta}}.\]
If
\[ |z| \leq 2^{-n_{L}\theta},\]
then there exists $n \geq n_{L}$ such that
\[ 2^{-(n+1)\theta} \leq |z| < 2^{-n\theta}.\] Thus,
\[|Lz| \geq 2^{-(n+1)} > \frac{1}{2} |z|^{\frac{1}{\theta}}.\]

We now put these two cases together to conclude that
\[|Lz| \geq C_{L} |z|^{\frac{1}{\theta}},\] where 
\begin{align}C_{L} & = \max\left\lbrace \frac{2^{-n_{L}}}{R^{\frac{1}{\theta}}}, 2^{-1}\right\rbrace. \qedhere
\end{align}
\end{proof}

There are a number of open questions that arise naturally from the results above. The most important are the following.
\begin{enumerate}
\item Can we extend Theorem \ref{ET5} without the restriction of the thickness exponent being less than $1$, in such a way that it improves on Theorem \ref{ET3}?
\item Can we prove the existence of a nonlinear bi--Lipschitz embedding into either a Euclidean or a Hilbert space when the thickness equals zero?
\end{enumerate}
\newpage

\section{Thickness and Dual Thickness}\label{S4}

In this section, we concentrate on the relation of the two exponents that were defined in the introduction in the context of a Hilbert and a Banach space. We note that we already know that the thickness is bounded above by the box--counting dimension.

We first give an immediate upper bound on the dual thickness based on yet another exponent.
\begin{definition}
Suppose that $\mathfrak{B}$ is a Banach space and $X \subset \mathfrak{B}$. 
Then, given any $\alpha >0$ and $\epsilon >0$ we denote by $m_{\alpha}(X,\epsilon)$ the smallest dimension of all those finite--dimensional subspaces $V$ of $\mathfrak{B}^{*}$ such that whenever $x,y \in X$ with $\|x-y\| \geq \epsilon$ there exists some $\Phi \in V$ with $\|\Phi\|=1$ that satisfies
\[|\Phi(x-y)| \geq \alpha \epsilon.\]
Then we define
\[\sigma_{\alpha}(X) = \limsup_{\epsilon \rightarrow 0}\frac{ \log m_{\alpha}(X,\epsilon)}{-\log\epsilon}.\]
\end{definition}

Following Robinson \cite{JCR}, we have the following estimate.
\begin{lemma}\label{COR}
Suppose that $\mathfrak{B}$ is a Banach space and $X \subset \mathfrak{B}$. Then
\[\tau^{*}(X) \leq \sigma_{\alpha}(X), \, \, \, \forall \, \, \alpha >0.\]
\end{lemma}

\begin{proof}
Let $\theta > 0.$ Let $V$ be a finite--dimensional subspace of $\mathcal{B}^{*}$ such that for all $x,y \in X$ with $\|x-y\| \geq \epsilon$, there exists $\phi \in V$ with $\|\phi\|=1$ and $|\phi(x-y)| \geq \alpha\epsilon$.
If $\epsilon$ is small enough such that $\epsilon^{\theta} < \alpha$, then
\[|\phi(x-y)| \geq \epsilon^{1+\theta},\]
which gives 
\begin{align}
\tau^{*}_{\theta}(X) & \leq \sigma_{\alpha}(X), \, \, \, \forall \, \, \theta >0. \qedhere
\end{align}
\end{proof}

We now prove that in a Hilbert space the dual thickness is always bounded above by the thickness. 

\begin{lemma}
Suppose $H$ is a Hilbert space and $X \subset H$, such that $\tau(X) < \infty$. Then 
\[\tau^{*}(X) \leq \tau(X).\]
\end{lemma}

\begin{proof}
Take $\epsilon >0$ and let $U$ be a finite--dimensional subspace of $H$ such that 
\[\dim(U) = d = d(X,\epsilon) \, \, \, \text{and} \, \, \, \dist(X,U) < \epsilon.\]

Now we consider $P$, the orthonormal projection onto $U$. For all $x \in X$ we have
\[\|x-Px\| = \dist(x,U) <\epsilon.\]
Let 
\[V = \{ L\circ P : L \in U^{*}\} \subset H^{*}.\]
It is easy to see that $V$ is finite--dimensional and that $\dim(V) = d$.

Suppose that $x,y \in X$ satisfy $\|x-y\| \geq \epsilon.$
Since $P(x-y) \equiv z \in U$, we define $L \colon U \to \mathbb{R}$ such that
\[L_{z}(u)= \frac{\langle u,z \rangle}{\|z\|},\]
for all $u \in U$.

Then, $\Phi = L_{z} \circ P \in V$ and $\|\Phi\| =1$. Moreover,
\[|\Phi(x-y)| = |L_{z}(z)| = \|z\| = \|P(x) - P(y)\| = \|x-y\| \geq \epsilon.\]
Therefore,
\begin{align}
\tau^{*}(X) & \leq \sigma_{1}(X) \leq \tau(X).\qedhere
\end{align}
\end{proof}

In the context of a Banach space, there is no known relationship between the thickness and the dual thickness. In the paper of Robinson \cite{Rob}, it is claimed that the dual thickness is bounded above by the box--counting dimension, but there is an error in the proof given there. However, we can prove that the dual thickness is bounded above by the box dimension of the Minkowski difference set $X-X$, which in particular is always bounded above by twice the box dimension of $X$.

\begin{lemma}
Suppose $\mathfrak{B}$ is a Banach space and $X \subset \mathfrak{B}$ compact. Then
\[\tau^{*}(X) \leq \romd_{B}(X-X) \leq 2\romd_{B}(X).\]
\end{lemma}

\begin{proof}
Let $Z = X-X$. 
Given $\epsilon >0$ and any $d > \romd_{B}(Z)$, we find $N=N(Z,\epsilon)\sim \epsilon^{-d}$ balls of radius $\epsilon$ with centres $z_{j}$ that cover $Z$. By the Hahn--Banach theorem, for any $j \leq N$, we obtain linear functionals $\phi_{j}$ that satisfy
\[\|\phi_{j}\| = 1 \, \, \text{and} \,\, |\phi_{j}(z_{j})|=\|z_{j}\|.\]
We now define $V = \lspan(\phi_{1},...,\phi_{N}).$

Suppose now that $z \in Z$ such that $\|z\| \geq 50 \epsilon$ and let $j \leq N$ such that $\|z-z_{j}\| < \epsilon.$ Thus,
\[\|z_{j}\| \geq 49\epsilon,\]
which gives
\[|\phi_{j}(z)| = |\phi_{j}(z-z_{j}) + \phi_{j}(z_{j})| \geq \|z_{j}\| - \epsilon \geq 48\epsilon.\]
This shows that
\[\sigma_{48/100}(X) \leq \romd_{B}(Z),\] and the conclusion is immediate by Lemma \ref{COR}.
\end{proof}

\subsection{`Orthogonal' sequences in $\ell_{p}$}
In the remainder of this paper, we concentrate on a particular subset of $\ell_{p}$, for $p \in [1,\infty]$, and prove that some of the inequalities we know so far are sharp.
These sets were first discussed by Ben Artzi et al \cite{BA} and have been used by Pinto De Moura \& Robinson \cite{DR} as examples to show that the H\"older exponent of the inverses in Hunt and Kaloshin's Theorem \ref{ET2} is asymptotically sharp.

Take $p \geq 1$ and let $(\alpha_{n})_{n=1}^{\infty}$ be a decreasing sequence  such that $\alpha_{n}\rightarrow 0.$ Then, for all $n$ let $e_{n} = (0,0,...,1,0,...)$ and define
\[A=\{a_{1},...,a_{n},...\} = \{\alpha_{1}e_{1},...,\alpha_{n}e_{n},...\}.\]
It is obvious that $a_{i} \in \ell_{p}$, for all $1\leq p < \infty,$ hence $A \subset \ell_{p}$. 
We also have $a_{i} \in c_{0}$, where $c_{0}$ is the space of real sequences converging to zero equipped with the $\ell_{\infty}$ norm. 
Following Robinson \cite{JCR}, we know that
\[\romd_{B}(A;\ell_{p}) = \limsup_{n \rightarrow \infty} \frac{\log n}{-\log\|a_{n}\|} = \inf\left\lbrace\nu > 0: \sum_{n=1}^{\infty} |a_{n}|^{\nu} < \infty\right\rbrace,\]
for all $p$. For the rest of this section, we implicitly understand the case $p = \infty$ as meaning $c_{0}$.

We first state without proof some additional properties of the box--counting dimension that we will need. The proofs can be found in Robinson \cite{JCR}.

\pagebreak

\begin{lemma}\label{BC2}\hfill
\begin{enumerate}
\item Let $\mathfrak{B}$ a Banach space and $A \subset \mathfrak{B}$ compact. Let $M(A,\epsilon)$ be the maximum number of points in $A$ that are $\epsilon$--separated, meaning that $\|x-y\| \geq \epsilon$, for any $x,y$ in that collection. Then
\[\romd_{B}(A) = \limsup_{\epsilon\rightarrow 0} \frac{\log M(A,\epsilon)}{-\log\epsilon}.\]
\item Suppose $(\mathfrak{B}_{1}, \| \cdot \|_{1})$ and $(\mathfrak{B}_{2}, \| \cdot \|_{2})$ are Banach spaces, $\mathfrak{B}_{1}\subseteq \mathfrak{B}_{2}$ and
\[\|u\|_{2} \leq C \|u\|_{1} ,\, \, \forall \,\, u \in \mathfrak{B}_{1}.\]
Then, for all compact subsets $X \subset \mathfrak{B}_{1}$,
\[\romd_{B}(X;\mathfrak{B}_{2})\leq \romd_{B}(X;\mathfrak{B}_{1}).\]
\end{enumerate}
\end{lemma}

We now prove that the inequality $\romd_{B}(X-X) \leq 2\romd_{B}(X)$ is sharp for this particular class of examples.

\begin{lemma}\label{BCA}
For all $p \geq 1,$
\begin{equation}
\romd_{B}(A-A;\ell_{p}) = 2\romd_{B}(A).
\end{equation}
\end{lemma}

\begin{proof}
We know that $\romd_{B}(A-A;\ell_{p}) \leq 2\romd_{B}(A).$
Hence, we need to show \[\romd_{B}(A-A;\ell_{p}) \geq 2\romd_{B}(A).\]  
We also have that $\ell_{p} \subset c_{0}$ and
\[\|u\|_{\infty} \leq \|u\|_{p},\]
for any $p < \infty.$
By part 2 of Lemma \ref{BC2}, we deduce that
\[\romd_{B}(A-A;\ell_{p})\geq \romd_{B}(A-A;c_{o}).\]
Thus, it suffices to prove that
\[\romd_{B}(A-A;c_{0}) \geq 2\romd_{B}(A).\]

Let $\epsilon >0.$ Take $N \in \mathbb{N}$, such that $\|a_{N}\|_{\infty} < \epsilon \leq \|a_{N-1}\|_{\infty}.$ 
Then
\[A \subseteq \bigcup_{i=0}^{N-1} B(x_{i},\epsilon),\]
where $x_{i} = a_{i} , \forall \, 1\leq i \leq N-1$ and $x_{0} = 0$. It is also obvious that $N=N(A,\epsilon)$. 
Set $z_{ij}=x_{i}-x_{j},$ for all $i\neq j$. 
Then, we claim that
\[A-A \subset \bigcup_{i,j=0}^{N-1} B(z_{ij},\epsilon) \bigcup B(0,\epsilon).\]

Take $z = a_{m}-a_{n} \in A-A,$ such that $\|z\|_{\infty} \geq \epsilon.$ Otherwise, $z \in B(0,\epsilon).$
Suppose, without loss of generality that $\|a_{m}\|_{\infty} \leq \|a_{n}\|_{\infty}.$Then, 
\[\|z\|_{\infty} = \|a_{n}\|_{\infty} \geq \epsilon,\] which implies that $n\leq N-1$ and $z_{n0} = x_{n}= a_{n}$.
We now have two cases: if $\|a_{m}\|_{\infty} < \epsilon,$ then
\[ \|a_{n} - a_{m} -x_{n}\|_{\infty} = \|a_{m}\|_{\infty}<\epsilon,\]
while if $\|a_{m}\|_{\infty} \geq \epsilon,$ then 
\[\|a_{n} - a_{m} - z_{nm}\|_{\infty}=0<\epsilon.\]

Now, let $M=M(A-A,\epsilon)$ denote the maximum number of $\epsilon$--separated points in $A-A$, i.e. 
\[\|y_{k}-y_{l}\|_{\infty} \geq \epsilon,\]
for all $y_{k},y_{l}$ in that collection.
Then, we claim that
\[\{z_{ij}\}_{i,j=0}^{N-1}\subseteq\{y_{k}\}_{k=1}^{M}.\]
Indeed suppose that for some $i,j ,i\neq j $, $z_{ij}= x_{i}-x_{j} = a_{i} -a_{j} \neq y_{k}, \, \forall \, k.$
Let $a_{n_{k}},a_{m_{k}}$ such that $y_{k} = a_{n_{k}}-a_{m_{k}}$ and assume wlog that $i\neq n_{k}$. Then, 
\begin{align*}
|z_{ij} - y_{k}\|_{\infty}& =\|a_{i}-a_{j}-a_{n_{k}}+a_{m_{k}}\| \\
& = \|\alpha_{i}e_{i}-\alpha_{j}e_{j}-\alpha_{n_{k}}e_{k}+\alpha_{m_{k}}e_{k}\|_{\infty}\\
& \geq \|\alpha_{i}e_{i}\| \geq \epsilon,
\end{align*} contradicting the fact that $M(A-A,\epsilon)$ is the maximum number of $\epsilon$--separated points.
All in all, we deduce that 
\[M(A-A,\epsilon) \geq (N(A,\epsilon)-1)^{2} -N(A,\epsilon) +1=N^{2}-3N+2.\]

Using part 1 of Lemma \ref{BC2} it follows that
\begin{align*}
\romd_{B}(A-A;\infty) & = \limsup_{\epsilon\rightarrow 0} \frac{\log M(A-A,\epsilon)}{-\log\epsilon}\\
& \geq  \limsup_{\epsilon\rightarrow 0} \frac{\log\left(N^{2}(A,\epsilon)-3N(A,\epsilon)+1\right)}{-\log\epsilon} \\
& \geq \limsup_{\epsilon\rightarrow 0} \frac{2\log N(A,\epsilon)}{-\log\epsilon} = 2\romd_{B}(A).\qedhere
\end{align*}
\end{proof}

We now show that the thickness exponent and box--counting dimension of this `orthogonal' sequence coincide whenever $p \leq 2$. 

We first make the straightforward remark that whenever $(\mathcal{B}_{1},\| \cdot \|_{1})\subseteq (\mathcal{B}_{2},\| \cdot \|_{2})$, $X \subset \mathcal{B}_{1}$ and 
\[\|u\|_{2} \leq C\|u\|_{1} \, \, \forall \, \, u \in \mathcal{B}_{1},\]
then 
\[\tau(X;\mathcal{B}_{2}) \leq \tau(X;\mathcal{B}_{1}).\]
In particular, since $\tau(X)\leq \romd_{B}(X)$, we only need to show that the thickness exponent of this orthogonal sequence equals the box--counting dimension when $p=2,$ This has already been proven in Robinson \cite{JCR}; here we give an alternative and slightly easier proof. We first need the following lemma.

\begin{lemma}
Take $\{\alpha_{n}\}_{n=1}^{\infty}$ as usual and let $A_{k}=\{a_{1},...,a_{k}\}= \{\alpha_{1}e_{1},...,\\ \alpha_{k}e_{k}\} \subset \ell_{2}.$
Then, 
\[d_{\ell_{2}}(A_{k},\epsilon) \geq k \left(1-\frac{\epsilon}{\|a_{k}\|}\right)^{2}.\]
\end{lemma}

\begin{proof}
We first remind ourselves that $d=d(A_{k},\epsilon)$ denotes the smallest dimension among those finite-dimensional subspaces $V$ of $\ell_{2}$ that satisfy
\[\mathrm{dist}(x,V) \leq \epsilon, \, \, \forall \, \, x \in A_{k}.\]

For all $i \leq k$ take $v_{i} \in V$ such that $\|v_{i}-a_{i}\|_{2} \leq \epsilon.$ Then,
\[\dim(\mathrm{span}(v_{1},...,v_{k})) = d.\]
Let $P$ denote the orthogonal projection onto $\mathrm{span}(v_{1},...,v_{k})$. Thus, we have
\[\|a_{i}-Pa_{i}\|_{2} = \mathrm{dist}(a_{i},\mathrm{span}(v_{1},...,v_{k})) \leq \|v_{i}-a_{i}\|_{2} \leq \epsilon.\]
Moreover,
\[\|a_{i}-Pa_{i}\| = |\alpha_{i}|\|e_{i}-P(e_{i})\| \geq |\alpha_{k}|(1-\|Pe_{i}\|),\] which gives
\[\|Pe_{i}\| \geq 1-\frac{\epsilon}{\|a_{k}\|}.\]
We know that for every orthogonal projection in a Hilbert space,
\[\mathrm{rank} \,P = \sum_{i=1}^{\infty} \|Pe_{i}\|^{2},\]
which proves that
\begin{align}
d & \geq k \left(1-\frac{\epsilon}{\|a_{k}\|}\right)^{2}.\qedhere
\end{align}
\end{proof}

We can now show that the thickness exponent and the box--counting dimension are equal in this case. For the proof, we use the argument in Robinson \cite{JCR} and the above lemma.
\begin{lemma}
\[\tau(A;\ell_{2}) = \romd_{B}(A) = \limsup_{n \rightarrow \infty} \frac{\log n}{-\log \|a_{n}\|}.\]
\end{lemma}

\begin{proof}
Take $n$ large enough such that $\|a_{n}\| < 1$ and take $n'\geq n$ such that
\[|\alpha_{n}|=|\alpha_{n+1}|=...=|\alpha_{n'}|>|\alpha_{n'+1}|.\]
Let 
\[\epsilon_{n} = \frac{\|a_{n}\| + \|a_{n'+1}\|}{4},\] which implies

\[\frac{\|a_{n}\|}{4} < \epsilon_{n} < \frac{|a_{n}\|}{2}.\]
By the previous lemma, we have
\[d(A,\epsilon) \geq d(A_{n'},\epsilon) \geq \frac{n'}{4} \geq \frac{n}{4} .\]

Therefore, we obtain
\begin{align*}
\tau(A) & \geq \limsup_{n \rightarrow \infty} \frac{d(A,\epsilon_{n})}{-\log \epsilon_{n}} \geq \limsup_{n \rightarrow \infty} \frac{\log n -\log 4}{\log 4- \log \|a_{n}\|} \\
& = \limsup_{n \rightarrow \infty} \left(\frac{\log n}{-\log\|a_{n}\|} \frac{1 - \frac{\log4}{\log n}}{1 - \frac{\log4}{\log \|a_{n}\|}}\right)\\
& = \limsup_{n \rightarrow \infty}\frac{\log n}{-\log\|a_{n}\|} = \romd_{B}(A).\qedhere
\end{align*}
\end{proof}

It still remains open whether the thickness exponent and box--counting dimension of this particular set coincide, whenever $p>2.$ However, we give a lower bound for the thickness exponent of $A$ that depends on the box--counting dimension and the conjugate exponent of $p$. 

\begin{lemma}\label{AP1}
Let $p \in [1,\infty]$ and $A \subset \ell_{p}$ as before. Then
\begin{equation}\label{AP}
\tau(A) \geq \frac{q \, \romd_{B}(A)}{q + \romd_{B}(A)},
\end{equation} where $q$ is the conjugate exponent of $p$.
\end{lemma}

We note that for $p=1$ the right hand side of \eqref{AP} becomes the box-counting dimension, giving a direct proof of what we proved in the previous lemma.

\begin{proof}[Proof of Lemma \ref{AP1}]
Suppose $a_{1},...,a_{k} \in A$. Consider $\epsilon_{k} = \frac{1}{2}\|a_{k}\|k^{-1/q}.$ Let $U$ be a subspace of $\ell_{p}$ such that $\dim(U)=d(\{a_{1},...,a_{k}\},\epsilon_{k})\leq d(A,\epsilon_{k})$
and let $v_{1},...,v_{k} \in U$ be such that
\[\|v_{i} - a_{i}\| \leq \epsilon_{k}.\]

We claim that $v_{1},,,,v_{k}$ are linearly independent and in particular that $k \leq \dim(U).$
Indeed, consider $\lambda_{i} \in \mathbb{R}$ such that 
\[ \sum_{i=1}^{k} \lambda_{i}v_{i} = 0.\]
Then

\begin{align*}
\epsilon_{k}\sum_{i=1}^{k} \left|\lambda_{i}\right| & \geq \sum_{i=1}^{k} \left|\lambda_{i} \right| \|v_{i} - a_{i}\|_{p} = \sum_{i=1}^{k_{0}}\| \lambda_{i} v_{i} - \lambda_{i} a_{i}\|_{p} \\
& \geq \| \sum_{i=1}^{k_{0}} \lambda_{i} v_{i} - \lambda_{i} a_{i}\|_{p}=
\|\sum_{i=1}^{k_{0}} \lambda_{i} a_{i}\|_{p}\\
& = \left(\sum_{i=1}^{k_{0}} \left| \lambda_{i} \alpha_{i} \right|^{p}\right)^{1/p}\geq |\alpha_{k}| \left(\sum_{i=1}^{k_{0}} |\lambda_{i}|^{p}\right)^{1/p}\\
& \geq k^{-1/q} \sum_{i=1}^{k} \left|\lambda_i \right|.
\end{align*}
Therefore,
\[\sum_{i=1}^{k} \left|\lambda_{i} \right| (- \frac{1}{2} k^{-1/q}) \geq 0,\] giving that $\lambda_{i} = 0$ for every $i$.

Thus,\[\romd(A,\epsilon_{k}) \geq k.\]
Now, we make the following computation
\begin{align*}
\tau(A) & \geq \limsup_{k \rightarrow \infty} \frac{\log \romd(A,\epsilon_{k})}{-\log\epsilon_{k}} \geq \limsup_{k \rightarrow \infty} \frac{\log k}{-\log(\|a_{k}\| k^{-1/q})}\\
& = \limsup_{k \rightarrow \infty} \frac{\log k}{1/q\log k - \log\|a_{k}\| }\\
& = \frac{1}{\liminf \left(\frac{1}{q} - \frac{\log\|a_{k}\|}{\log k}\right)} = \frac{1}{\frac{1}{q}+\frac{1}{\romd_{B}(A)}}\\
& = \frac{q \, \romd_{B}(A)}{q + \romd_{B}(A)}.\qedhere
\end{align*}
\end{proof}

As mentioned in the beginning of the section, we do not know if the dual thickness is always bounded above by the box--counting dimension. However, Pinto de Moura and Robinson \cite{DR} relied on this fact to prove that the dual thickness and box--counting dimension of these orthogonal sequences coincide for every $p \in [1,\infty]$. In the next lemma, we show that this upper bound is true in this particular case.
\begin{lemma}Suppose that $A=\{a_{1},...,a_{k},\dots\}\subset\ell_{p}$ is as usual. Then,
\[\tau^{*}(A) \leq \romd_{B}(A), \, \, \mbox{for any} \, \, p \, \in [1,\infty].\]
\end{lemma}

\begin{proof}
Take any $\epsilon > 0$. Then, there exists some $N=N(A,\epsilon)$, such that $|\alpha_{N}| < \epsilon$ and $|\alpha_{N-1}|\geq \epsilon.$ Then, $\forall i \leq N$ let $x \in \ell_{p}$ and set $\phi_{i}(x) \in \mathbb{R}$ to be the $i$-th coordinate of $x$.
Since we are just projecting on one direction we immediately obtain that $\|\phi_{i}\|=1$ and
\[\phi_{i}(a_{j}) = \delta_{ij}\alpha_{i}.\]
Let
\[V = \mathrm{span}(\phi_{1},...,\phi_{N}) \subset \ell_{p^{*}}.\]

Now, take any $a_{n},a_{m}$ in $A$ such that $\|a_{m}\|_{p}\leq\|a_{n}\|_{p}$ and $\|a_{n}-a_{m}\|_{p} \geq 50\epsilon.$ Let $i_{n},i_{m}\leq N$ such that
\[\|a_{n}-a_{i_{n}}\|_{p} < \epsilon \, \, \, \text{and} \, \, \, \|a_{m}-a_{i_{m}}\|_{p} < \epsilon.\]
Since $\|a_{m}\|_{p}\leq\|a_{n}\|_{p}$ and $\|a_{n}-a_{m}\|_{p} \geq 50\epsilon,$ we obtain $\|a_{n}\|_{p} \geq 25\epsilon$ which gives $\|a_{i_{n}}\|_{p} \geq  24\epsilon.$

Set 
\[\Phi = \frac{\phi_{i_{n}} - \phi_{i_{m}}}{\|\phi_{i_{n}} - \phi_{i_{m}}\|},\] and we have
$\Phi \in V, \|\Phi\|=1$ and
\begin{align*}
|\Phi(a_{n}-a_{m})| & \geq \frac{|(\phi_{i_{n}} - \phi_{i_{m}})(a_{n}-a_{i_{n}}+a_{i_{n}}-a_{i_{m}}+a_{i_{m}}-a_{m})|}{2}\\
& \geq \frac{20\epsilon}{2}= 10\epsilon.
\end{align*}

Arguing as in Lemma \eqref{BCA}, we obtain
$\tau^{*}(A) \leq \sigma_{1/5}(A) \leq \romd_{B}(A).$
\end{proof}

All in all, there is no known relation between the thickness and the dual thickness in the context of a Banach space. Moreover, it is still open whether or not the dual thickness is always bounded above by the box dimension.

\section{Conclusion}
Using a combination of the method introduced by Hunt and Kaloshin and the argument of Robinson \cite{Rob}, we established an embedding theorem which does not require the use of the dual thickness.
It still remains open whether we can extend the result without the restriction that the thickness is less than one. Moreover, we discussed some interesting properties of a particular subset of $\ell_{p}$ which is used by Pinto de Moura \& Robinson \cite{DR} as an example to prove that Hunt and Kaloshin's theorem for the Hilbert space case is asymptotically sharp. 

We note that all the results we stated and proved here require a linear embedding which somehow restricts the regularity that we can achieve for the inverse. A significant open problem is whether we can find an embedding of a subset $X$ of either a Banach or a Hilbert space, that is bi--Lipschitz, but not necessarily linear. An example in \cite{JCR} shows that this is not always possible for subsets with finite box--counting dimension.

However, it is known that a necessary but not sufficient condition for such an embedding to exist is that $X$ must have finite Assouad dimension (see \cite{Ass}) which can be conceived as a more local version of the box--counting dimension. A method similar to that we presented here was used by Olson and Robinson \cite{OR} to provide `almost' bi--Lipschitz linear embeddings of subsets with finite Assouad dimension. This dimension is also discussed in more recent work by Gottlieb, Lee \& Krauthgamer \cite{GLK} and by Naor \& Neiman \cite{NN}, who establish nonlinear embedding theorems.
\pagebreak

\end{document}